\documentclass[12pt]{amsart}
\usepackage{hyperref}
\usepackage{graphicx,amsmath,amsfonts,latexsym,amssymb,amsthm}
\usepackage{latexsym,pictexwd,dcpic}

  \newcommand{\field}[1]{\mathbb{#1}}
  
  \newcommand{\cbb}{\field{C}}

  \newcommand{\C}{{\mathbb C}}

  \newcommand{\I}{\,\mathrm{i}\,}
  \newcommand{\PsDO}{\Psi\mathrm{DO}}

  \newcommand{\beq}{\begin{equation}}
  \newcommand{\eeq}{\end{equation}}

  \newtheorem{theorem}{Theorem}[section]

  \newtheorem{pro}[theorem]{Proposition}
  \newtheorem{cor}[theorem]{Corollary}
  \newtheorem{rem}[theorem]{Remark}
\begin{document}

\title[Spectral Properties of K\"ahler manifolds]
{On the Spectrum of geometric operators on K\"ahler manifolds}

\author[D. Jakobson]{Dmitry Jakobson}
\address{Department of Mathematics and Statistics, McGill
University, 805 Sherbrooke Str. West, Montr\'eal QC H3A 2K6,
Ca\-na\-da.} \email{jakobson@math.mcgill.ca}

\author[A. Strohmaier]{Alexander Strohmaier}
\address{Department of Mathematical Sciences,  Loughborough University,  Loughborough, Leicestershire, LE11 3TU,
UK} \email{a.strohmaier@lboro.ac.uk}

\author[S. Zelditch]{Steve Zelditch}
\address{Johns Hopkins University, Department of Mathematics, 404 Krieger
Hall, 3400 N. Charles Street, Baltimore, MD 21218}
\email{zelditch@math.jhu.edu}



\date{\today}

\thanks{The first author was supported by NSERC, FQRNT and Dawson fellowship.}

\maketitle

\begin{abstract}
 On a compact K\"ahler manifold there is a canonical action
 of a Lie-superalgebra on the space of differential forms. It is generated
 by the differentials, the Lefschetz operator and the adjoints of these operators.
 We determine the asymptotic distribution of irreducible representations
 of this Lie-superalgebra on the eigen\-spaces of the Laplace-Beltrami operator.
 Because of the high degree of symmetry the Laplace-Beltrami operator on forms
 can not be quantum ergodic. We show that after taking these symmetries into
 account quantum ergodicity holds for the Laplace-Beltrami operator and
 for the Spin$^{\cbb}$-Dirac operators if the unitary frame flow is ergodic. The
 assumptions for our theorem are known to be satisfied
 for instance for negatively curved K\"ahler manifolds of odd complex dimension.
\end{abstract}

%-------------------------------------------------------------------------------
%                       INTRODUCTION AND OVERVIEW
%-------------------------------------------------------------------------------

\section{Introduction}

 Properties of the spectrum of the Laplace-Beltrami operator on a manifold are closely related
 to the properties of the underlying classical dynamical system. For
 example ergodicity of the geodesic flow on the unit tangent bundle $T_1 X$ of a compact
 Riemannian manifold $X$ implies quantum ergodicity.
 Namely, for any complete orthonormal sequence of eigenfunctions $\phi_j \in
 L^2(X)$ to the Laplace operator $\Delta$ with eigenvalues $\lambda_j \nearrow
\infty$ one has (see \cite{Shn74,Shn93,Zel87,CdV,HMR})
\begin{gather}
 \lim_{N \to \infty} \frac{1}{N} \sum_{j \leq N} | \langle \phi_j,
 A \phi_j \rangle - \int_{T_1^*X} \sigma_A(\xi) d \mu_L(\xi)|^2 =0,
\end{gather}
for any zero order pseudodifferential operator $A$,
where integration is with respect to the normalized Liouville measure $\mu_L$ on the
unit-cotangent bundle $T_1^* X$, and $\sigma_A$ is the principal symbol of
$A$. Quantum ergodicity is equivalent to the existence of a subsequence
$\phi_{j_k}$ of counting density one such that
\begin{gather}
 \lim_{k \to \infty} \langle \phi_{j_k}, A \phi_{j_k} \rangle = \int_{T_1^*X} \sigma_A(\xi) d \mu_L(\xi).
\end{gather}
In particular, $A$ might be a smooth function on $X$ and the above implies that
the sequence
\begin{gather}
  |\phi_{j_k}(x)|^2 dV_g
\end{gather}
converges to the normalized Riemannian measure $d V_g$ in the weak topology
of measures. For bundle-valued geometric operators like the Dirac
operator acting on sections of a spinor bundle or the
Laplace-Beltrami operator the corresponding Quantum ergodicity for
eigensections is known in a precise way to relate to the ergodicity
of the frame flow on the corresponding manifold \cite{JS}; see also
\cite{BG04.1,BG04.2,BO06}.

This paper deals with a situation in which the frame flow is not ergodic,
namely the case of K\"ahler manifolds. In this case the conclusions in \cite{JS}
do not hold since there is a huge symmetry algebra acting on the space of
differential forms. This algebra is the universal enveloping algebra of a
certain Lie superalgebra that is generated by the Lefschetz operator, the
complex differentials and their adjoints. On the level of harmonic forms this symmetry
is responsible for the rich structure of the cohomology of K\"ahler manifolds
and can be seen as the main ingredient for the Lefschetz theorems. Here we are
interested in eigensections with non-zero eigenvalues, that is in the spectrum
of the Laplace-Beltrami operator acting on
the orthogonal complement of the space of harmonic forms. The action of the
Lie superalgebra on the orthogonal complement of the space of harmonic forms
is much more complicated than on the space of harmonic forms  where it
basically becomes the action of $sl_2(\cbb)$. In this paper we classify
all finite dimensional unitary representations of this algebra and determine
the asymptotic distribution of these representations in the eigenspaces.
Since the typical irreducible representation of the algebra decomposes into
four irreducible representation for $sl_2(\cbb)$ this shows that
eigenspaces to the Laplace-Beltrami operator have multiplicities.
An important observation in our treatment is that the universal enveloping algebra of this
Lie superalgebra is
generated by two commuting subalgebras, one of which is isomorphic to the universal
enveloping algebra of $sl_2(\cbb)$. This $sl_2(\cbb)$-action is generated
by an operator $L_t$ and its adjoint $L_t^*$ which is going to be defined in
section \ref{tlef}. This operator can be interpreted as the Lefschetz operator
in the directions of the frame bundle which are orthogonal to the frame flow.
However, $L_t$ is not an endomorphism of vector bundles, but it acts as a
pseudodifferential operator
of order zero.

Guided by this result we tackle the question of quantum ergodicity for the
Laplace-Beltrami operator on $(p,q)$-forms. Unlike in the case of ergodic
frame flow it turns out that there might be different quantum limits of
eigensections on the space of co-closed $(p,q)$-forms because of the presence
of the Lefschetz operator. Our main results establishes quantum ergodicity
for the Dirac operator and the Laplace Beltrami operator if one takes the
Lefschetz symmetry into account and under the assumption that the $U(m)$-frame
flow is ergodic. For example our analysis shows that in case of an
ergodic $U(m)$-frame flow for any
complete sequence of co-closed primitive $(p,q)$-forms there is a density
one subsequence which converges to a state which is an extension of the
Liouville measure and can be explicitly given.
For the Spin$^\cbb$-Dirac operators we show that quantum ergodicity
does not hold for K\"ahler manifolds of complex dimension greater than one.
Thus, negatively curved
Spin-K\"ahler manifolds provide examples of manifolds with ergodic geodesic flow
where quantum ergodicity does not hold for the Dirac operator.
Our analysis shows that there are certain invariant subspaces for the Dirac operator
in this case and we prove quantum ergodicity for the Dirac operator restricted to these subspaces
provided that the $U(m)$-frame flow is ergodic.

\section{K\"ahler manifolds}

Let $(X,\omega,J)$ be a K\"ahler manifold of real dimension $n=2m$. Let $g$ be the metric,
$h=g + \I \omega$ be the hermitian metric, and $\omega$ the symplectic form.
As usual let $J$ be the complex structure.
A $k$-frame $(e_1,\ldots,e_k)$ for the cotangent space at some point $x \in X$
is called unitary if it is unitary with respect to the hermitian inner product
induced by $h$. Hence, a $k$-frame $(e_1,\ldots,e_k)$ is unitary iff
$(e_1,J e_1,e_2,J e_2,\ldots,e_k,J e_k)$ is orthonormal with respect to $g$.
A unitary $m$-frame at a point $x \in X$ is an ordered orthonormal basis
for $T_x^* X$ viewed as a complex vector space.

Clearly, the group $U(m)$
acts freely and transitively on the set of unitary $m$-frames. The bundle
$U_m X$ of unitary $m$-frames is therefore a $U(m)$-principal fiber bundle.
Let $T_1^* X$ be the unit cotangent bundle with bundle projection $\pi$.
Then projection onto the first vector makes $U_m X$ a principal $U(m-1)$-bundle
over $T_1^* X$.
{\footnotesize
$$
\begindc{0}[7]
  \obj(10,15)[A]{$U_m X$}
  \obj(40,15)[B]{$X$}
  \obj(25,15)[C]{$T_1^* X$}
  \mor{A}{C}{$U(m-1)$}
  \mor{C}{B}{$S^{2m-1}$}
  \cmor((10,13)(11,11)(15,10)(25,10)(35,10)(39,11)(40,13)) \pup(25,8){$U(m)$}
\enddc
$$
}
Transporting covectors parallel with respect to the Levi-Civita
connection extends the Hamiltonian flow on $T_1^* X$ to a flow on $U_m X$
which we call the $U(m)$-frame flow (in the literature it is also referred to
as the restricted frame flow). This is indeed a flow on $U_m X$
since $J$ is covariantly constant and therefore unitary frames are parallel
transported into unitary ones. This flow is the appropriate replacement for
the $SO(2m)$-frame flow for K\"ahler manifolds as it can be shown to be ergodic
in some cases, whereas the $SO(2m)$-frame flow never is ergodic for K\"ahler manifolds

Suppose that $X$ is a negatively-curved K\"ahler manifold
\cite{Borel}.  We summarize results that can be found in \cite{Br82,
BrG, BrP}. We refer the reader to \cite{BuP, JS, Br82} and
references therein for discussion of frame flows on general
negatively-curved manifolds. Note that the frame flow is {\em not}
ergodic on negatively-curved K\"ahler manifolds, since the almost
complex structure $J$ is preserved. This is the only known example
in negative curvature when the geodesic flow is ergodic, but the
frame flow is not. In fact, given an orthonormal $k$-frame
$(e_1,\ldots,e_k)$, the functions $(e_i,J e_j), 1\leq i,j\leq k$ are
first integrals of the frame flow.

However, the following proposition was proved in \cite{BrG}:
\begin{pro}\label{erg:flow}
Let $X$ be a compact negatively-curved K\"ahler manifold of {\em
complex} dimension $m$.  Then the $U(m)$-frame flow is ergodic on
$U_m X$ when $m=2$, or when $m$ is odd.
\end{pro}

\section{The Hodge Laplacian and the Lefschetz decomposition}

Let $\wedge^* X$ be the complex vector bundle $\wedge^* T_\cbb^* X$,
where $T_\cbb^* X$ is the complexification of the co-tangent bundle.
Then the Lefschetz
operator $L : C^\infty(X;\wedge^* X) \to C^\infty(X;\wedge^* X)$
is defined by exterior multiplication with the K\"ahler form $\omega$, i.e.
$L=\omega \wedge$. Its adjoint $L^*$ is then given by interior multiplication
with $\omega$. Is is well known that
\begin{gather}
 [L^*,L]:=H=\sum_k (m-r) P_r,
\end{gather}
where $P_r$ is the orthoprojection onto $C^\infty(\wedge^r X)$,
and $L,L^*,H$ define a representation of $sl_2(\cbb)$ which commutes
with the Laplace operator $\Delta=2\Delta_{\bar \partial}=2\Delta_{\partial}$.
The decomposition into irreducible representations on the level of harmonic
forms is called the Lefschetz decomposition. We will refer to this decomposition
as the Lefschetz decomposition in general.
Note that since the Lefschetz operator commutes with $\Delta$ each eigenspace
decomposes into a direct sum of irreducible subspaces for the
$sl_2(\cbb)$ action.

The operators $L,L^*,H,\partial,\bar\partial,\partial^*,\bar\partial^*,\Delta$
satisfy the following relations (see e.g. \cite{Ballmann})
\begin{gather}
 [L,\bar\partial^*]=-i \partial,\; [L^*,\partial]=i \bar \partial^*,\; \nonumber
 [L^*,\bar \partial]=-i \partial^*,\; [L,\partial^*]=i \bar \partial,\\ \nonumber
 [L^*,L]=H, \; [H,L]=-2 L,\; [H,L^*]=2 L^*,\\ \nonumber
 \{\partial,\partial\}= \{\bar\partial,\bar\partial\}=
 \{\partial^*,\partial^*\}= \{\bar\partial^*,\bar\partial^*\}=0, \\ \label{superrelations}
 [L,\bar\partial]=[L,\partial]=[L^*,\bar\partial^*]=[L^*,\partial^*]=0,\\ \nonumber
 \{\partial,\bar\partial\}=\{\partial,\bar\partial^*\}= \{\bar\partial,\partial^*\}=0,\\ \nonumber
 \{\partial,\partial^*\}= \{\bar\partial,\bar\partial^*\}=\frac{1}{2}\Delta. \nonumber
\end{gather}

Thus, the operators form a Lie superalgebra with central element $\Delta$
(see also \cite{FrGrRe99}).
Let $\Delta^{-1}|_{\mathrm{ker} \Delta^\perp}$ the inverse of the Laplace operator
on the orthocomplement of the kernel of $\Delta$. We view this as an operator defined
in $L^2(X,\wedge^* X)$ by defining it to be zero on $\mathrm{ker}\Delta$ and write
$\Delta^{-1}$ slightly abusing notation.

\subsection{The transversal Lefschetz decomposition} \label{tlef}

The operator $Q:=2 \Delta^{-1} \bar \partial \partial$ is a partial isometry with
initial space $\mathrm{Rg}(\bar \partial^*) \cap \mathrm{Rg}(\partial^*) $ and final
space $\mathrm{Rg}(\bar \partial) \cap \mathrm{Rg}(\partial)$. Hence,
$Q^* Q$ is the orthoprojection onto  $\mathrm{Rg}(\bar \partial^*) \cap \mathrm{Rg}(\partial^*) $
and $QQ^*$ is the orthoprojection onto $\mathrm{Rg}(\bar \partial) \cap
\mathrm{Rg}(\partial)$. From the above relations one gets
\begin{gather}
 [L,Q]=0,\\
 [L,Q^*]=2 i \Delta^{-1}(\bar \partial \bar \partial^* - \partial^* \partial),\\
 [Q^*,Q]=-2 \Delta^{-1}(\bar \partial \bar \partial^* - \partial^* \partial),
\end{gather}
from which one finds that
\begin{gather}
 [L-iQ,Q^*]=[L-iQ,Q]=0.
\end{gather}
We define the transversal Lefschetz operator $L_t$ by
\begin{gather}
 L_t:=L-i Q.
\end{gather}
Then clearly $L_t^*=L^*+i Q^*$ and one gets that
\begin{gather}
 [L_t^*,L_t]=H_t=H+[Q^*,Q],\\
 [H_t,L_t]=-2 L_t, \quad [H_t,L_t^*]=-2 L_t^*,\\
 [\partial,L_t]=[\partial^*,L_t]=[\bar\partial,L_t]=[\bar\partial^*,L_t]=0,
\end{gather}
and hence, also the transversal Lefschetz operators defines an action of
$\mathrm{sl}_2(\cbb)$ on $L^2(X,\wedge^* X)$. Unlike the Lefschetz
operator the transversal Lefschetz operator commutes with the holomorphic
and antiholomorphic codifferentials.

Denote by $\mathfrak{g}$ the Lie-superalgebra generated by
$a,\bar a,L,H,a^*,\bar a^*,L^*$ and relations
\begin{gather}
 [L,\bar a^*]=-i a,\; [L^*,a]=i \bar a^*,\;\nonumber
 [L^*,\bar a]=-i a^*,\; [L,a^*]=i \bar a,\\ \nonumber
 [L^*,L]=H, \; [H,L]=-2 L,\; [H,L^*]=2 L^*,\\ \nonumber
 \{a,a\}= \{\bar a,\bar a\}= \label{superrelations2}
 \{a^*,a^*\}= \{\bar a^*,\bar a^*\}=0,\\
 [L,\bar a]=[L,a]=[L^*,\bar a^*]=[L^*, a^*]=0,\\ \nonumber
 \{\bar a,a\}=\{\bar a,a^*\}= \{a,\bar a^*\}=0,\\\nonumber
\{a,a^*\}= \{\bar a,\bar a^*\}=1.
\end{gather}
The subspace of odd elements is spanned by $a,a^*,\bar a, \bar a^*$, the
subspace of even elements is spanned by $L,L^*$ and $H$.
In the following we will denote by $\mathcal{U}(\mathfrak{g})$ the universal enveloping algebra of this
Lie-superalgebra viewed as a unital $*$-algebra, i.e. the unital $*$-algebra generated
by the symbols $\{L,L^*,H,a,a^*,\bar a,\bar a^*\}$ and the above relations.

The relations (\ref{superrelations2}) are obtained from the relations
(\ref{superrelations}) by
sending $a$ to $\sqrt{2} \Delta^{-1/2} \partial$ and $\bar a$ to $\sqrt{2}
\Delta^{-1/2} \bar \partial$. Therefore, we obtain a $*$-representation of the
Lie-superalgebra $\mathfrak{g}$
on the orthogonal complement of the kernel of $\Delta$.

\subsection{The representation theory of $\mathcal{U}(\mathfrak{g})$}

The calculations in the previous section used the relations in
$\mathcal{U}(\mathfrak{g})$ only. Hence, they remain valid if we
regard $Q=\bar a a$ and $L_t:=L-i Q$ as elements in the abstract
$*$-algebra $\mathcal{U}(\mathfrak{g})$. Hence, $L_t,L_t^*$ generate
a subalgebra in $\mathcal{U}(\mathfrak{g})$ which is canonically
isomorphic to the universal enveloping algebra of $\mathrm{sl}_2(\cbb)$
and which we therefore denote by $\mathcal{U}(\mathrm{sl}_2(\cbb))$. Note that
$\mathcal{U}(\mathrm{sl}_2(\cbb))$ commutes with $a,\bar a, a^*$ and $\bar a^*$.
Since
$\mathcal{U}(\mathfrak{g})$ is generated by two commuting
subalgebras the representation theory for
$\mathcal{U}(\mathfrak{g})$ is very simple. The $*$-subalgebra
$\mathcal{A}$ generated by $a$ and $\bar a$ has the following
canonical representation on $\wedge^{*} \cbb^2 \cong
\cbb^4$. For an orthonormal basis $\{e,\bar e\}$ of $\cbb^2$ define
the action of $a$ by exterior multiplication by $\I e$, and the
action of $\bar a$ by exterior multiplication by $\I \bar e$. It is
easy to see that all non-trivial finite dimensional irreducible
$*$-representations of $\mathcal{A}$ are unitarily equivalent to
this representation.

Note that the equivalence classes of finite dimensional irreducible $*$-representations of
$\mathcal{U}(\mathrm{sl}_2(\cbb))$ are labeled by the non-negative integers.
Denote the Verma-module for the Spin-$\frac{n}{2}$ representation by $V_n$
and the distinguished highest weight vector in $V_n$ by $h$.
Remember that $V_n$ is spanned by vectors of the form $L_t^k h$ with $k=0,\ldots,n$
and we have $L_t^* h=0$ and $H_t h = n h$.

Now define an action of $\mathcal{U}(\mathfrak{g})$ on
$H_n:=V_n \otimes \wedge^{*,*} \cbb^2$ by
\begin{gather}\nonumber
 L_t (v \otimes w) = (L_t v) \otimes w,\\ \nonumber
 L_t^* (v \otimes w) = (L_t^* v) \otimes w,\\\nonumber
 H_t (v \otimes w) = (H_t v) \otimes w,\\
 a (v \otimes w)= v \otimes a w,\\\nonumber
 \bar a (v \otimes w)= v \otimes \bar a w,\\\nonumber
 a^* (v \otimes w)= v \otimes a^* w,\\\nonumber
 \bar a^* (v \otimes w)= v \otimes \bar a^* w,\nonumber
\end{gather}

Clearly, this defines a $*$-representation
of $\mathcal{U}(\mathfrak{g})$ on $H_n$.

\begin{theorem}
 The representations $H_n$ are irreducible and pairwise
 inequivalent. Any non-trivial finite dimensional irreducible $*$-representation
 of $\mathcal{U}(\mathfrak{g})$ is unitary equivalent to some $H_n$.
\end{theorem}
\begin{proof}
 Since $\mathcal{U}(\mathfrak{g})$ is generated by two commuting subalgebras
 $\mathcal{U}(\mathrm{sl}_2(\mathbb{C}))$ and $\mathcal{A}$ any
 irreducible $*$-representation of is also an irreducible $*$-representation
 of $\mathcal{U}(\mathrm{sl}_2(\mathbb{C})) \otimes \mathcal{A}$. If it is
 finite dimensional it is therefore a tensor product of two finite dimensional
 irreducible representations of $\mathcal{U}(\mathrm{sl}_2(\mathbb{C}))$ and $\mathcal{A}$.
\end{proof}

\begin{cor}
 Any non trivial finite dimensional irreducible
 $*$- representation of $\mathcal{U}(\mathfrak{g})$
 decomposes into a direct sum of $4$ equivalent modules for the $\mathrm{sl}_2(\mathbb{C})$
 action defined by $L_t,L_t^*,H_t$.
\end{cor}

If $h_n$ is a highest weight vector of $V_n$ then the kernel of $L_t^*$
in the representation $H_n$ is given by $h_n \otimes \wedge^* \cbb^2$. Using
the unitary basis $e,\bar e$ for $\cbb^2$ as before we see that the vectors
$$h_n \otimes 1, h_n \otimes e, h_n \otimes \bar e $$ are in the kernel of
$L^*$. Moreover,
\begin{gather}
 H (h_n \otimes 1) = (n-1)(h_n \otimes 1),\\
 H (h_n \otimes e) =  n (h_n \otimes e),\\
 H (h_n \otimes \bar e) =  n (h_n \otimes \bar e).
\end{gather}
Therefore, in the decomposition of $H_n$ into irreducibles of the
$\mathrm{sl}_2(\mathbb{C})$ action defined by $L,L^*,H$ the representations
$V_n$ occur with multiplicity at least $2$ and the representation $V_{n-1}$
occurs with multiplicity at least $1$.
The vector $h_n \otimes (e \wedge \bar e)$ has weight $n+1$ and therefore,
there must be another representation of highest weight greater or equal than
$n+1$ occurring.
Since
$$
 4 \dim V_{n} - 2 \dim V_n - \dim V_{n-1}=\dim V_{n+1}
$$
this shows that as a module for the $\mathrm{sl}_2(\mathbb{C})$
action defined by $L,L^*,H$ we have $H_n=V_{n+1} \oplus V_{n} \oplus  V_{n} \oplus V_{n-1}$.

\begin{cor}\label{simplificator1}
 Every non-trivial finite dimensional irreducible $*$- representation of $\mathcal{U}(\mathfrak{g})$
 is as a module for the $\mathrm{sl}_2(\mathbb{C})$
 action defined by $L,L^*,H$ unitarily equivalent to the direct sum
 $V_{n+1} \oplus V_{n} \oplus  V_{n} \oplus V_{n-1}$. By convention
 $V_{-1}=\{0\}$.
\end{cor}

\begin{cor} \label{simplificator}
 Let $V$ and $W$ be two finite dimensional $\mathcal{U}(\mathfrak{g})$ modules.
 Then $V$ and $W$ are unitarily equivalent if and only if they are equivalent as modules
 for the $sl_2(\mathbb{C})$ action defined by $L,L^*,H$.
\end{cor}

\subsection{The model representations}

There is another very natural representation $\rho$ of the $*$-algebra
$\mathcal{U}(\mathfrak{g})$ which is important for our purposes. This representation will be referred to
as the model representation and can be described as follows.
Let us view $\cbb^{m} \cong \mathbb{R}^{2m}$ as a real
vector space with complex structure $J$.
Let $\{e_i\}_{i=1,\ldots,m}$ be the standard unitary basis in $\mathbb{R}^{2m}$. Then  in
the complexification $\mathbb{R}^{2m} \otimes \cbb = \mathbb{C}^{2m}$
we define
\begin{gather}
 w_i= e_i - \I J e_i,\\
 \bar w_i=e_i + \I J e_i.
\end{gather}
We define $\rho(L)$ to be the operator of exterior multiplication by
$\omega=\frac{\I}{2}\sum_{i=1}^m w_i \wedge \bar w_i$ on the space $\wedge^* \cbb^{2m} =
\bigoplus_{p,q} \wedge^{p,q} \cbb^{2m}$. Let $\pi(L^*)$ be its adjoint, namely the
operator of interior multiplication by $\omega$. Let $\rho(a)$
be the operator of exterior multiplication by
$\frac{\I}{\sqrt{2}} w_1$ and
$\rho( \bar a)$ be the operator of exterior multiplication by
$\frac{\I}{\sqrt{2}} \bar w_1$.
The operators $\rho(a^*)$ and $\rho(\bar a^*)$ are defined as the adjoints
of $\rho(a)$ and $\rho(\bar a)$.
This defines a representation $\rho$
of $\mathcal{U}(\mathfrak{g})$ on $\wedge^* \cbb^{2m}$.
This representation
decomposes into a sum of irreducibles. Note that $\rho(L_t)=\rho(L-\I \bar a a)$ is given by
exterior multiplication by $\frac{\I}{2}\sum_{i=2}^m w_i \wedge \bar w_i$.
The restriction of $\rho$ to the two subalgebras generated by
$\rho(L),\rho(L^*),\rho(H)$ and $\rho(L_t),\rho(L^*_t),\rho(H_t)$ define
representations of $sl_2(\mathbb{C})$. Since the maximal eigenvalue of $H$ is $m$,
only representations of highest weight $k$ with $k \leq m$ can occur in the decomposition of the model
representation with respect to the $sl_2(\mathbb{C})$-action by $\rho(L),\rho(L^*),\rho(H)$.
Consequently, by Cor  \ref{simplificator1}  in the decomposition of the model representation into irreducible representations
only the representations $H_k$ with $k \leq m$ can occur.

\section{Asymptotic decomposition of Eigenspaces}

Since the action of $\mathcal{U}(\mathfrak{g})$ commutes with the Laplace
operator $\Delta$ on forms each eigenspace
$$V_{\lambda} = \{\phi \in \wedge^* X: \Delta \phi = \lambda
\phi\}$$ with $\lambda \not=0$
is a $\mathcal{U}(\mathfrak{g})$-module and can be decomposed into
a direct sum of irreducible $\mathcal{U}(\mathfrak{g})$-modules.
In the previous section we classified all irreducible $*$-representations
of $\mathcal{U}(\mathfrak{g})$ and found that they are isomorphic
to $H_k$ for some non-negative integer $k$.
Therefore, we may define the function $m_k(\lambda)$ as
\begin{gather}
 m_k(\lambda):=\{\mbox{the multiplicity of} \;
 H_k \; \mbox{in}\;\; V_{\lambda} \},
\end{gather}
so that
\begin{gather}
 V_\lambda \cong \bigoplus_{k = 0}^{\infty} m_k(\lambda) H_k
\end{gather}

\begin{theorem}\label{main1}
Let $X$ be any compact K\"ahler manifold of complex dimension $m$.
Then in the decomposition of the eigenspaces of the Laplace-Beltrami operator $\Delta$ into irreducible
representations of $\mathcal{U}(\mathfrak{g})$
the proportion of irreducible summands of
type $H_k$ in $L^2(X; \wedge^* X)$ is in average the same as the
proportion of such irreducibles in the model representation of
$\mathcal{U}(\mathfrak{g})$ on $\wedge^* \C^{2m}$:
\begin{equation}\label{asymp:mult}
\frac{1}{N(\lambda)} \sum_{j: \lambda_j \leq \lambda} m_k(\lambda_j)
\sim  \frac{1}{\mathrm{dim}(\wedge^*{\cbb^{2m}})} m_k(\wedge^* \C^{2m}) ,
\end{equation}
 where $N(\lambda) = \mathrm{tr} \Pi_{[0,\lambda]}$ and
$\Pi_{[0, \lambda]}$ is the spectral projection of the Laplace-Beltrami operator $\Delta$.
\end{theorem}

We recall that $N(\lambda) \sim \frac{rk(E) vol(X)}{(4 \pi)^m
\Gamma(m + 1)} \lambda^{m}$ for the Laplacian on a bundle $E \to X$
of rank $rk(E)$ over a manifold $X$ of real dimension $2m$.
Note that apart from the fact that we are not dealing with a group
but with a Lie superalgebra the action
is neither on $X$, nor on $T^*X$, but rather on
the total space of the vector bundle $\pi^*(\wedge^* X) \to T_1^* X$ . The action there
leaves the fibers invariant and therefore it is rather different from a group action
on the base manifold. The above theorem
thus falls outside the scope of the equivariant Weyl laws of
articles such as \cite{BH1, BH2, GU, HR, TU}. In fact its conclusion is rather different from the conclusions in these
articles as in our case only a fixed number
of types of irreducible representations may occur.

\begin{proof}
 For a compact K\"ahler manifold  $\mathcal{U}(\mathfrak{g})$ acts
by pseudodifferential operators on $C^\infty(X;\wedge^* X)$. Therefore, the symbol
map defines an action of $\mathcal{U}(\mathfrak{g})$ on each fiber of the bundle
$\pi^* (\wedge^* X) \to T_1^* X$. The representation of $\mathcal{U}(\mathfrak{g})$ on each
fiber is easily seen to be equivalent to the model representation.
 Since the maximal eigenvalue of $H$, acting on $L^2(X;\wedge^* X)$, is $m$, only representations
 of highest weight $k$ with $k \leq m$ can occur in the decomposition of
 $L^2(X;\wedge^* X)$ into irreducible subspaces
 with respect to the $\mathrm{sl}_2(\cbb)$-action by $L^*,L,H$. Again, by
 Cor \ref{simplificator1} types $H_k$ with $k > m$
 cannot occur in the decomposition with respect to the
 $\mathcal{U}(\mathfrak{g})$-action.
 Let $P_k$ be the orthogonal projection onto the type $H_k$ in $L^2(X;\wedge^* X)$.
 Then $P_k$ is actually a pseudodifferential operator of order $0$.
 Namely, the quadratic Casimir operator $\mathcal{C}$ of the
 $\mathrm{sl}_2(\cbb)$-action by $L^*_t,L_t,H_t$ given by
 \begin{gather}
  \mathcal{C} = L^*_t L_t + L_t L_t^* + \frac{1}{2} H^2,
 \end{gather}
 is a pseudodifferential operator of order $0$. On a subspace
 of type $H_k$ it acts like multiplication by $\frac{k^2}{2}+k$.
 Therefore, if $Q$ is a real polynomial that is equal to $1$
 at $\frac{k^2}{2}+k$ and equal to $0$ at $\frac{l^2}{2}+l$ for any integer $l \not=k$
 between $0$ and $m$ it follows that $P_k=Q(\mathcal{C})$.
 Thus, $P_k$ is a pseudodifferential operator of order $0$ and its principal
 symbol at $\xi$ projects onto the subspace in the fiber $\pi^*_\xi( \wedge^* X)$ which is spanned
 by the representations of type $H_k$. Therefore, for every $\xi$:
 \begin{gather}
  \frac{1}{\mathrm{dim}(H_k)} \mathrm{tr}(\sigma_{P_k}(\xi))=m_k(\wedge^* \C^{2m}).
 \end{gather}
 Applying Karamatas Tauberian theorem to the heat trace expansion
 \begin{gather}
  \mathrm{tr}(P_k e^{- t \Delta}) = (4\pi)^{-m} \mathrm{Vol}(X)
  \left(\int_{T_1^* X} \mathrm{tr}(\sigma_{P_k}(\xi)) d\xi \right) t^{m} + O(t^{m-\frac{1}{2}}).
 \end{gather}
 gives
 \begin{gather}
  \frac{1}{N(\lambda)} \sum_{j: \lambda_j \leq \lambda}
   \mathrm{tr}(\Pi_{[0,\lambda]} P_k) \sim m_k(\wedge^* \C^{2m})
   \mathrm{dim}(H_k) \frac{1}{\mathrm{dim}(\wedge^*{\cbb^{2m}})}.
 \end{gather}
 After dividing by $\mathrm{dim}(H_k)$ this reduces to the statement of the theorem.
\end{proof}

\begin{rem}
A natural question is whether, for generic K\"ahler metrics, the eigenspaces of
the Laplace-Beltrami operator are irreducible representations of the Lie superalgebra  $\mathfrak{g}$
and of complex conjugation.
Such irreducibility  is suggested by the    heuristic principle of  `no accidental degeneracies',
i.e. in generic cases,  degeneracies of  eigenspaces should   be  entirely due to  symmetries
(see  \cite{Zel90} for some results  and references).
Cor. \ref{cor53} would then suggest that for a  generic K\"ahler manifold the
spectrum of $\Delta$ on the space of primitive co-closed $(p,q)$-forms
should be  simple for fixed $p$ and $q$.
\end{rem}

\section{Quantum ergodicity for the Laplace-Beltrami operator}

We will now investigate the question of quantum ergodicity for the Laplace-Beltrami operator
on a compact K\"ahler manifold $X$ and we keep the notations from the previous sections.
As shown in  \cite{JS} this question  is intimately related to the ergodic decomposition
of the tracial state on the $C^*$-algebra $C(X;\pi^*\wedge ^* X)$. The transversal Lefschetz
decomposition plays an important role here.

\subsection{Ergodic decomposition of the tracial state}

On the space of $(p,q)$-forms denote by $P_{p,q}$ the projection onto the space
of transversally-primitive forms, i.e. onto the kernel of $L_t^*$. Let $P_{p,q,k}$
be the projection onto the range of $L_t^k P_{p-k,q-k}$.
The operators
\begin{gather}
P_1=P_{\partial \bar \partial}=4 \Delta^{-2} \partial \bar \partial \bar \partial^* \partial^*=Q Q^*,\\
P_2=P_{\partial^* \bar \partial^*}=4 \Delta^{-2}\partial^* \bar \partial^* \bar \partial \partial=Q^*Q,\\
P_3=P_{\partial \bar \partial^*}=4 \Delta^{-2}\partial \bar \partial^* \bar \partial \partial^* ,\\
P_4=P_{\partial^* \bar \partial}=4 \Delta^{-2}\partial^* \bar \partial \bar \partial^* \partial
\end{gather}
are projections onto the ranges of the corresponding operators.
We have
\begin{gather}
 P_H+\sum_{i=1}^4 P_i=1
\end{gather}
 where $P_H$ is the finite dimensional projection onto the space of harmonic forms.
 Using the transversal Lefschetz decomposition we obtain a further decomposition
 \begin{gather}
  \sum_{k=0}^{\mathrm{min}(p,q)} P_{p,q,k} P_H + \sum_{k=0}^{\mathrm{min}(p,q)}\sum_{i=1}^4
  P_{p,q,k} P_i=1
 \end{gather}
 where each of the subspaces onto which $P_{p,q,k} P_i$ projects is invariant under the Laplace
 operator.

Note that the principal symbols of these projections are invariant projections
 in $C(T_1^*X,\pi^*\mathrm{End}(\wedge^{p,q} X))$ and the above relation gives  rise to a
 decomposition of the
 tracial state $\omega_{tr}$ on $C(T_1^*X,\pi^*\mathrm{End}(\wedge^{p,q} X))$ defined by
 \begin{gather}
  \omega_{tr}(a):=\frac{1}{\mathrm{rk}(\wedge^{p,q} X)} \int_{T_1^*X} \mathrm{tr}(a(\xi)) d\xi
 \end{gather}
into invariant states. Thus, the tracial state is not ergodic. However,
if the $U(m)$-frame flow is ergodic this decomposition turns out to be ergodic.
\begin{pro} \label{proerg}
 Suppose that the $U(m)$-frame flow on $U_m X$ is ergodic. Let $P$ be
 one of the projections
 \begin{gather*}
   P_{p,q,k} P_i,\\
   1 \leq i \leq 4,\\
   0 \leq k \leq \mathrm{min}(p,q).
 \end{gather*}
 Then the state $\omega_P$ on $C(T_1^*X;\pi^{*}\mathrm{End}(\wedge^{p,q} X))$
 defined by $\omega_P(a):=c_P \omega_{tr}(\sigma_P a)$ is ergodic.
 Here $c_P=\omega_{tr}(\sigma_P)^{-1}$.
\end{pro}
\begin{proof}
 The bundle $\wedge^{p,q} X$ can be naturally identified with the associated
 bundle $U_m X \times_{\hat \rho_1} \wedge^{p,q} \mathbb{C}^{2m}$, where $\hat\rho_1$ is the
 representation of $U(m)$ on
 $$
  \wedge^{p,q} \mathbb{C}^{2m} = \wedge^p \mathbb{C}^m \otimes \wedge^q \overline{\mathbb{C}}^m.
 $$
 obtained from the canonical representation on $\mathbb{C}^m$.
 The pull back $\pi^* \wedge^{p,q} X$ of $\wedge^{p,q} X$ can analogously be
 identified with the associated bundle
 \begin{gather}
  U_m X \times_{\hat \rho} \wedge^{p,q} \mathbb{C}^{2m},
 \end{gather}
 where $\hat \rho$ is the restriction of $\hat \rho_1$ to the subgroup $U(m-1)$.
 Since the first vector in $\mathbb{C}^m$ is invariant under the action of
 $U(m-1)$ we have the decomposition
 $$
  \wedge^{p,q} \mathbb{C}^{2m} = \wedge^{p,q} \mathbb{C}^{2m-2} \oplus
  \wedge^{p-1,q} \mathbb{C}^{2m-2} \oplus \wedge^{p,q-1} \mathbb{C}^{2m-2}
  \oplus \wedge^{p-1,q-1} \mathbb{C}^{2m-2}
 $$
 into invariant subspaces. The projections onto these subspaces in each
 fiber is exactly given by the principal symbols $\sigma_{P_i}$ of the
 projections $P_i$. The representation of $U(m-1)$ on $\wedge^{p',q'} \mathbb{C}^{2m-2}$
 may still fail to be irreducible. However, it is an easy exercise in representation theory
 (c.f. \cite{MR1153249}, Exercise 15.30, p. 226) to show that the kernel of
 $\sigma_{L_t^*}$ in each fiber is an irreducible representation of $U(m-1)$. Thus,
 $\sigma_P$ projects onto a sub-bundle $F$ of $\pi^* \wedge^{p,q} X$ that is associated
 with an irreducible representation $\rho$ of $U(m-1)$, i.e.
 \begin{gather}
  F \cong U_m X \times_\rho V_\rho.
 \end{gather}
 This identification intertwines the $U(m)$-frame flow on $U_m X$ and the flow $\beta_t$.
 To show that the state $\omega_P$ is ergodic it is enough to show that
 any positive $\beta_t$-invariant element $f$ in
 $\sigma_P L^\infty(T_1^*X,\pi^* \mathrm{End} \wedge^{p,q} X) \sigma_P$
 is proportional to $\sigma_P$ (see \cite{JS}, Appendix). Under the above
 identification $f$ gets identified with a function $\hat f \in L^\infty(U_m X; V_\rho)$
 which satisfies
 \begin{gather}
  \hat f(x g) =\rho(g) \hat f(x) \rho(g)^{-1}, \quad x \in U_m X, g \in U(m-1).
 \end{gather}
 If such a function is invariant under the $U(m)$-frame flow it follows
 from ergodicity of the $U(m)$-frame flow
 that it is constant almost everywhere. So almost everywhere $\hat f(x)= M$,
 where $M$ is a matrix. By the above transformation rule $M$ commutes
 with $\rho(g)$. Since $\rho$ is irreducible it follows that $M$ is a multiple
 of the identity matrix. Thus, $\hat f$ is proportional to the identity
 and consequently, $f$ is proportional to $\sigma_P$.
\end{proof}

Applying the abstract theory developed in \cite{MR1384146} the same argument
as in \cite{JS} can be applied to obtain

\begin{theorem}
 Let $P$ be one of the projections
 \begin{gather*}
   P_{p,q,k} P_i,\\
      1 \leq i \leq 4,\\
   0 \leq k \leq \mathrm{min}(p,q).
 \end{gather*}
 and let $(\phi_j)$ be an orthonormal basis in $\mathrm{Rg}(P)$ with
 \begin{gather}
  \Delta \phi_j = \lambda_j \phi_j,\\ \nonumber
  \lambda_j \nearrow \infty.
 \end{gather}
 If the $U(m)$-frame flow on $U_m X$is ergodic, then
 quantum ergodicity holds in the sence that
 \begin{gather}
  \frac{1}{N} \sum_{j=1}^N |\langle \phi_j,A \phi_j \rangle - \omega_P(\sigma_A)| \to 0,
 \end{gather}
 for any $A \in \PsDO_{cl}^0(X,\wedge^{p,q} X)$.
\end{theorem}

Since for co-closed forms primitivity and transversal primitivity are
equivalent there is a natural gauge condition that manages without the above
heavy notation.
\begin{cor} \label{cor53}
 Let $\phi_j$ be a complete sequence of primitive co-closed $(p,q)$-forms
 such that
 \begin{gather}
  \Delta \phi_j = \lambda_j \phi_j,\\ \nonumber
  \lambda_j \nearrow \infty.
 \end{gather}
 Then, if the $U(m)$-frame flow on $U_m X$is ergodic,
 quantum ergodicity holds in the sence that
 \begin{gather}
  \frac{1}{N} \sum_{j=1}^N |\langle \phi_j,A \phi_j \rangle - \omega_P(\sigma_A)| \to 0,
 \end{gather}
 for any $A \in \PsDO_{cl}^0(X,\wedge^{p,q} X)$, where $P=P_{p,q,0} P_2$ is
 the orthogonal projection onto the space of primitive co-closed $(p,q)$-forms.
\end{cor}

\section{Quantum ergodicity for Spin$^\cbb$-Dirac operators}

In this section we consider the quantum ergodicity for  Dirac type operators rather than Laplace operators.
The complex structure on K\"ahler manifolds gives rise to the so-called canonical and anti-canonical Spin$^\cbb$-
structures. The spinor bundle of the latter can be canonically identified with the bundle $\wedge^{0,*} X$
in such a way that the Dirac operator gets identified with the so-called Dolbeault Dirac operator. Other Spin$^\cbb$-
structures (e.g. the canonical one) can then be obtained by twisting with a holomorphic line bundle. Let us quickly describe
the construction of the twisted Dolbeault operator.

Let $L$ be a holomorphic line bundle. Then the twisted Dolbeault complex is given by
$$
\begindc{0}[7]
 \obj(15,10)[A]{$\ldots$}
 \obj(25,10)[B]{$\wedge^{0,k-1}X \otimes L$}
 \obj(40,10)[C]{$\wedge^{0,k}X \otimes L$}
 \obj(50,10)[D]{$\ldots$}
 \mor{A}{B}{$\bar \partial$}
 \mor{B}{C}{$\bar \partial$}
 \mor{C}{D}{$\bar \partial$}
\enddc
$$
This is an elliptic complex and
the  twisted Dolbeault Dirac operator is defined by
\begin{gather}
 D=\sqrt{2}(\bar \partial + \bar \partial^*).
\end{gather}
As mentioned above this operator is the Dirac operator of a Spin$^{\cbb}$-structure
on $X$ where the spinor bundle is identified with  $S=\wedge^{0,*} X \otimes L$. Note that
Spin structures on $X$ are in one-one correspondence
with square roots of  the canonical bundle $K=\wedge^{n,0} T X$, i.e.
with holomorphic line bundles $L$ such that $L \otimes L = K$.
In this case the Dirac operator $D$ is
exactly the twisted Dolbeault Dirac operator.

The twisted Dolbeault Dirac operator is a first order elliptic formally self-adjoint differential operator. It is therefore
self-adjoint on the domain $H^1(X;\wedge^{0,*} X \otimes L)$ of sections in the first Sobolev space. As $D$
is a first order differential operator its spectrum is unbounded from both sides.

The Dolbeault Laplace operator is  given by $2(\bar \partial \bar \partial^* +
\bar \partial^* \bar \partial) =  D^2$ and will be denoted by $\Delta^L$.
The Hodge decomposition is
\begin{gather}
 C^\infty(X;\wedge^{0,k} X \otimes L)=\\ \nonumber
 \mathrm{ker}(\Delta^L_k) \oplus
 \bar \partial  C^\infty(X;\wedge^{0,k-1} X \otimes L) \oplus \bar
 \partial^*  C^\infty(X;\wedge^{0,k+1} X \otimes L).
\end{gather}
Note that the Dirac operator leaves $\mathrm{ker}(\Delta^L_k)$
invariant since it commutes with $\Delta^L$. Moreover, $D$ maps $\bar
\partial  C^\infty(X;\wedge^{0,k-1} X \otimes L)$ to $\bar
\partial^*  C^\infty(X;\wedge^{0,k} X \otimes L)$ and $\bar
\partial^*  C^\infty(X;\wedge^{0,k} X \otimes L)$ to $\bar \partial
C^\infty(X;\wedge^{0,k-1} X \otimes L)$. Therefore, the subspaces
\begin{gather}
 \mathcal{H}^k=\bar \partial  C^\infty(X;\wedge^{0,k-1} X \otimes L)
  \oplus \bar \partial^*  C^\infty(X;\wedge^{0,k} X \otimes L)
\end{gather}
are invariant subspaces for the Dirac operator. The orthogonal
projections $\Pi_k$ onto the closures of $\mathcal{H}^k$ are clearly zero
order pseudodifferential operators which commute with the Dirac operator.

Let $\overline{\PsDO_{cl}^0(X;S)}$ be the norm closure of the
$*$-algebra of zero order pseudodifferential operators in $\mathcal{B}(L^2(X,S))$.
Then the symbol map extends to an isomorphism
\begin{gather}
 \overline{\PsDO_{cl}^0(X;S)}/\mathcal{K} \cong C(T_1^* X,\pi^* \mathrm{End}(S)).
\end{gather}
By theorem 1.4 in \cite{JS} $\overline{\PsDO_{cl}^0(X;S)}$ is invariant
under the automorphism group $\alpha_t(A):=e^{-i (\Delta^L)^{1/2} t} A e^{+i (\Delta^L)^{1/2} t}$
and the induced flow $\beta_t$ on $C(T_1^* X,\pi^* \mathrm{End}(S))$
is the extension of the geodesic flow defined by parallel translation along the fibers.

As in the analysis for the Laplace-Beltrami operator we have to consider the tracial state
\begin{gather}
 \omega_{tr}(a)=\frac{1}{\mathrm{rk}(S)} \int_{T_1^*X} \mathrm{tr}(a(\xi)) d \xi,
\end{gather}
As already remarked in \cite{JS} this state
is not ergodic with respect to $\beta_t$ since it has a decomposition
\begin{gather} \label{decompspec}
 \omega_{tr}(a)=\frac{1}{2} \omega_+(a) + \frac{1}{2} \omega_-(a),
\end{gather}
where
\begin{gather}
 \omega_\pm(a)=\omega_{tr}((1 \pm \sigma_{\mathrm{sign}(D)}) a)
\end{gather}
On Spin$^{\cbb}$-manifolds with ergodic frame flows the states
$\omega_\pm$ were shown in \cite{JS} to be ergodic. On K\"ahler manifolds
of complex dimension  greater than one they are not ergodic since we have a further decomposition
\begin{gather}
 \omega_\pm(a)=\sum_{k} \omega_\pm(\sigma_{\Pi_k} a)
\end{gather}
into invariant states.

\begin{pro}\label{pro:dirac}
 Suppose that the $U(m)$-frame flow on $U_m X$ is ergodic. Then the states
 $\omega_\pm^k:= c_k \omega_\pm(\sigma_{\Pi_k} a)$
 are ergodic with respect to the group $\beta_t$. Here $c_k:=\omega_\pm(\sigma_{\Pi_k})^{-1}$.
\end{pro}
\begin{proof}
 Let $R$ be one of the projections $\frac{1 \pm \mathrm{sign}(D)}{2}\Pi_k$ and let
 $\sigma_R$ be its principal symbol. Hence, $\sigma_R$ is a projection in
 \begin{gather}
  C(T_1^*X,\pi^* \mathrm{End}(S)) \cong C(T_1^*X,\pi^* \mathrm{End}(\wedge^{0,*}X)).
 \end{gather}
 We need to show that $a \to \omega_{tr}(\sigma_R)^{-1}\omega_{tr}(\sigma_R a)$
 is an ergodic state. As in the proof of Proposition \ref{proerg} this is equivalent to showing that
 any positive element in $L^\infty(T_1^*X,\pi^*\mathrm{End}(\wedge^{0,k}X)) \sigma_R$
 is proportional to $\sigma_R$.
 A positive element in $L^\infty(T_1^*X,\pi^* \mathrm{End}(\wedge^{0,k}X)) \sigma_R$
 is also in $$\sigma_R L^\infty(T_1^*X,\pi^* \mathrm{End}(\wedge^{0,k}X)) \sigma_R=
 L^\infty(T_1^*X,\mathrm{End} F),$$ where $F$ is the sub-bundle of $\pi^*\wedge^{0,k}X$
 onto which $\sigma_R$ projects. Since $\sigma_R$ is $\beta_t$-invariant the flow
 clearly restricts to a flow on the sub-bundle $\mathrm{End}F$ of
 $\pi^* \mathrm{End}(\wedge^{0,k}X)$. We will show that
 under the stated assumptions an invariant element in $L^\infty(T_1^*X,\mathrm{End}F)$
 is proportional to the identity in $L^\infty(T_1^*X,\mathrm{End}F)$.
 Note that $\pi^* \wedge^{0,k} X$ is naturally identified with an associated
 bundle
 \begin{gather}
  \pi^* \wedge^{0,k} X \cong U_m X \times_{\wedge^k \tilde \rho} \wedge^k
  \overline \cbb^{m},
 \end{gather}
 where $\tilde \rho$ is the restriction of the anticanonical representation
 of $U(m)$ on $\bar \cbb^m$ to $U(m-1)$. Here we view $U_m X$
 as a $U(m-1)$-principal fiber bundle over $T_1^* X$.
 Note that $\wedge^k \tilde \rho$ is not irreducible but splits into a direct
 sum of two irreducible representations. This corresponds to the splitting
 $\wedge^k (\bar\cbb^{m-1} \oplus \bar\cbb)=\wedge^{k-1} \bar\cbb^{m-1} \oplus
 \wedge^{k} \bar\cbb^{m-1}$. Under the above correspondence the projections
 onto the sub-representations are exactly the principal symbols of the projections
 onto $\mathrm{Rg}(\bar \partial)$ and $\mathrm{Rg}(\bar \partial^*)$. One
 finds that $F$ is associated with a representation $\rho$ of $U(m-1)$
 \begin{gather}
  F \cong U_m X \times_{\rho} V_\rho,
 \end{gather}
 where
 $\rho$ is equivalent to $\wedge^{k} \hat \rho$ and $\hat \rho$ is the anticanonical
 representation of $U(m-1)$. Therefore, $\rho$ is irreducible.
 Hence, elements in $f  \in L^\infty(T_1^*X,\mathrm{End}F)$ can be identified
 with functions $\hat f \in L^\infty(U_m X,\mathrm{End}V_\rho)$ that satisfy the
 transformation property
 \begin{gather}
  \hat f(x g)=\rho(g) \hat f(x) \rho(g)^{-1}, \quad x \in U_m X, g \in U(m-1).
 \end{gather}
 This identification intertwines the pullback of the frame flow with
 $\beta_t$. Now exactly in the same way as in the proof of Proposition
 \ref{proerg} we conclude that an invariant element in $L^\infty(T_1^*X,\mathrm{End}F)$
 must be a multiple of the identity. Thus, the corresponding state is ergodic.
\end{proof}

The above theorem gives rise to an ergodic decomposition of the tracial state
on the $C^*$-algebra of continuous sections of $\pi^* \mathrm{End}(S)$
which is different from the decomposition obtained from Prop. \ref{proerg}. The advantage
of this decomposition is that it is more suitable to study quantum ergodicity for the
Dirac operator. Namely, the decomposition (\ref{decompspec}) corresponds to the
splitting into negative energy and positive energy subspaces of the Dirac operator. Thus, if we
are interested in quantum limits of eigensections with positive energy we need to
decompose the state $\omega_+$ into ergodic components. This is achieved by Prop. \ref{pro:dirac}.

In the same way as in \cite{JS} one obtains
\begin{theorem}
 Let $X$ be a K\"ahler manifold and let $L$ be a holomorphic line bundle.
 Let $D$ be the associated $Spin^{\cbb}$-Dirac operator and let $L^2_+(X,S)$
 be the positive spectral subspace of $D$. Suppose that $(\phi_j)$ is an orthonormal
 basis in $\Pi_k L^2_+(X,S)$ such that
 \begin{gather}
  D \phi_j = \lambda_j \phi_j,\\ \nonumber
  \lambda_j \nearrow \infty.
 \end{gather}
 If the $U(m)$-frame flow on $U_m X$ is ergodic, then
 \begin{gather}
  \frac{1}{N} \sum_{j=1}^N |\langle \phi_j,A \phi_j \rangle - \omega_k(\sigma_A)| \to 0,
 \end{gather}
 for any $A \in \PsDO_{cl}^0(X,S)$. Here $\omega_k$ is the state on
 $C(T_1^* X,\pi^* \mathrm{End}(S))$ defined by
 \begin{gather}
  \omega_k(a)= C \int_{T_1^*X} \mathrm{tr}
  \left((1+\sigma_{\mathrm{sign}(D)}(\xi))\sigma_{\Pi_k}(\xi) a(\xi)\right) d \xi,
 \end{gather}
 where integration is with respect to the normalized Liouville measure
 and $C$ is fixed by the requirement that $\omega_k(1)=1$.
\end{theorem}

This shows that quantum ergodicity for the Dirac operators holds only after
taking the symmetry $\Pi_k$ into account.
The states $\omega_k$ differ for different $k$. Therefore, Dirac operators on a K\"ahler manifolds
of complex dimension greater than one are never quantum ergodic in the sense of \cite{JS}.

\end{document}